\documentclass[11pt]{article}

\usepackage{geometry}
\usepackage{mathtools}

\usepackage{graphicx}
\usepackage{amsmath,amssymb,amsthm}
\usepackage{enumerate}
\usepackage{hyperref}
\usepackage[active]{srcltx}
\usepackage[T1]{fontenc}
\usepackage{color}
\usepackage[makeroom]{cancel}

\newtheorem{theo}{Theorem}

\newtheorem{lem}{Lemma}[section]

\newtheorem{cor}[lem]{Corollary}
\newtheorem{prop}[lem]{Proposition}

\newcommand{\eps}{\varepsilon}
\newcommand{\qtext}[1]{\quad\mbox{#1}\quad}
\newcommand{\qqtext}[1]{\qquad\mbox{#1}\qquad}

\newcommand{\rd}{\mathrm d}
\newcommand{\R}{\mathbb{R}}

\newcommand{\g}{\gamma}
\renewcommand{\l}{\lambda}
\renewcommand{\b}{\beta}
\newcommand{\EE}{\mathbb E}
\newcommand{\PP}{\mathbb P}
\renewcommand{\O}{\Omega}
\newcommand{\G}{\Gamma}

\newcommand{\indic}{\mathbf{1}}
\renewcommand{\L}{\mathcal L}

\numberwithin{equation}{section}
\begin{document}
\title{Trivariate distribution of sticky Brownian motion}
\date{}
\author{Jean-Baptiste Casteras and L\'eonard Monsaingeon}

\maketitle
\abstract{
In this short note we derive a closed form for the trivariate distribution (position, local time at the origin, and positive occupation time) of the one-dimensional sticky Brownian motion, thereby filling some gaps and fixing some mistakes in the literature.
}
%%%%%%%%%%%%%%%%%%%%%%%%%%%%%%%%%%%%%%%%%%%%%%%%%%%%%%%%%%%%%%%%%%%%%%%%%%%%%%%%%%%%%%%%%%
\section{Introduction}
Sticky diffusions are stochastic processes that behave as normal diffusions in the interior of a domain but also can also ``stick'' to the boundary for a positive amount of time (possibly following a different dynamics there), while being allowed to reenter after hitting.
The simplest example of such diffusions is the one-dimensional, $\R$-valued \emph{sticky Brownian motion} $\left(S_t\right)_{t\geq 0}$ (SBM in short), which sticks at $x=0$.
Feller introduced and studied SBM in the 1950's \cite{feller1952parabolic,feller1954diffusion,feller1957generalized} while trying to classify all possible behaviours of one-dimensional diffusion at a boundary from a purely functional-analytic point of view.
Sticky diffusions have recently attracted renewed interest, see e.g. \cite{anagnostakis2022path,peskir2015boundary,engelbert2014stochastic,T21,amir1991sticky,bou2020sticky,konarovskyi2021spectral} and references therein.

The goal of this paper is to derive carefully a closed form for the full trivariate SBM distribution, where trivariate refers here to
\begin{enumerate}
 \item
position $y\in \R$
\item
local time at the origin
\begin{equation}
\label{eq:Lt=Lt}
L_t(S)\coloneqq \lim\limits_{\eps\to 0}\frac 1{2\eps}\int_0^t \indic_{\{[-\eps,+\eps]\}}(S_u)\rd u
=\theta\int_0^t\indic_{\{S_u=0\}}\rd u
\end{equation}
\item
positive occupation time
\begin{equation}
\label{eq:def_Gammat}
\Gamma_t(S)\coloneqq \int_0^t \indic_{\{S_u\geq 0\}}\rd u.
\end{equation}
\end{enumerate}
(Here $\theta>0$ is a stickiness coefficient.)
The second equality in \eqref{eq:Lt=Lt} is far from trivial and is actually part of the requirements to be satisfied by SBM, see Section~\ref{sec:SBM} for details.
Before stating our main result let us fix notations and define (for $t\geq 0$)
\begin{equation}
 \label{eq:def_h}
g(t,x)=g_t(x)\coloneqq \frac{1}{\sqrt{2\pi t}} e^{-\frac{|x|^2}{2t}}
\qqtext{and}
 h(t,x)=h_t(x)\coloneqq \frac{|x|}{\sqrt{2\pi}t^{\frac 32}} e^{-\frac{|x|^2}{2t}}.
\end{equation}
These are the Gaussian kernel and the distribution of the first hitting time $T_0$ at the origin for standard Brownian motion, respectively.
For $x,y\geq 0$ we write
$$
p^0_t(x,y)
=\Big[g(t,x-y)-g(t,x+y)\Big]
=\frac{1}{\sqrt{2\pi t}}\left[e^{-\frac{|x-y|^2}{2t}} - e^{-\frac{|x+y|^2}{2t}}\right]
$$
for the transition kernel of killed Brownian motion, i-e $\PP_x(B_t\in\rd y,T_0<t)=p^0_t(x,y)\rd y$.
Our main result reads
\begin{theo}
\label{theo:trivariate_x}
Let $(S_t)_{t\geq 0}$ be a standard SBM started from $x\geq 0$, with local time $L_t=L_t(S)$ and occupation time $\Gamma_t=\Gamma_t(S)$.
Then the trivariate distribution is given by
\begin{multline}
\label{eq:trivariate_x}
\PP_x(S_t\in\rd y,\,\G_t\in\rd\tau,\,L_t\in \rd l)\\
=
p^0_t(x,y)\rd y\delta_t(\rd\tau)\delta_0(\rd l)
+\frac 1\theta h\left(\tau-\frac l\theta,\frac l2\right)h\left(t-\tau,\frac l2+x\right)\delta_0(\rd y)\rd\tau\rd l
\\
+\rd y\rd\tau\rd l
\begin{cases}
h\left(\tau-\frac l\theta,\frac l2+y\right)h\left(t-\tau,\frac l2+x\right) & \mbox{if }y\geq 0\\
h\left(\tau-\frac l\theta,\frac l2\right)h\left(t-\tau,\frac l2+x-y\right) & \mbox{if }y\leq 0
\end{cases}.
\end{multline}
\end{theo}
\noindent
Note that we implicitly mean here $0\leq \frac l\theta \leq \tau\leq t$ since $\frac{L_t}{\theta}=\int_0^t\indic_{\{S_u=0\}}\rd u\leq \int_0^t \indic_{\{S_u\geq 0\}}\rd u =\Gamma_t\leq t$.
We also stress that this full trivariate distribution \eqref{eq:trivariate_x} contains a Dirac measure $\delta_0(\rd y)$.
This fills a gap in \cite[Thm. 3.3]{T21} (with $\alpha=\frac 12$, zero-skewness), where only the absolute continuous $\rd y$ part for $y\neq 0$ was derived.
% \\

In a future work \cite{CMN} we will study sticky diffusions from a variational Partial Differential Equations perspective.
We will show that the corresponding Fokker-Planck equations, coupling a standard parabolic equation in the interior of a domain $\Omega\subset \R^d$ with Laplace-Beltrami driven diffusion along its boundary $\partial\Omega$ \cite{konarovskyi2021spectral}, can be recovered as gradient flows of nonstandard entropy functionals with respect to some Wasserstein-like ``sticky optimal transport'' distances in the spirit of \cite{jordan1998variational}.
Related bulk/interface transportation distances were previously introduced and exploited by the authors in \cite{chalub2021gradient,monsaingeon2021new,casteras2022hidden}, but unfortunately do not provide the correct framework for the problem at hand in \cite{CMN}.
The right transportation model will instead be dictated by purely stochastic considerations, which will require for technical reasons a closed form of the ($y,L_t$) distribution of \emph{reflected} Sticky Brownian motion.
This \emph{bivariate} distribution will be established here as an immediate consequence of Theorem~\ref{theo:trivariate_x}, but for later bibliographical convenience we record here the corresponding statement as
\begin{theo}
\label{theo:bivariate_reflected_x}
The one-sided, reflected SBM $\bar S_t\coloneqq |S_t|$ started from $x\geq 0$ has bivariate distribution
\begin{multline}
\label{eq:bivariate_reflected_x}
\PP_x(\bar S_t\in\rd y,\,L_t\in \rd l)
\\
=
p^0_t(x,y)\rd y\delta_0(\rd l)
+
2 h\left(t-\frac l\theta, l+x+y\right)\rd y\rd l
+
\frac 1\theta h\left(t-\frac l\theta,l+x\right)\delta_0(\rd y)\rd l.
\end{multline}
\end{theo}
\noindent
Here we implicitly mean $y\geq 0$, and again $0\leq \frac l\theta\leq t$.
\\

Let us briefly sketch the main line of proof for Theorem~\ref{theo:trivariate_x}.
We follow the strategy of \cite{karatzas1984trivariate}, later generalized in \cite{APP11,T21}.
The first step consists in computing the $(t,\tau,l)$ Laplace transform of the joint law $\PP_0(B_t\geq y,\,\G_t\in\rd\tau,\,L_t\in \rd l)$ for Brownian motion, started from $x=0$ and for a fixed parameter $y\in\R$.
 Then, exploiting the fact that BM and SBM are connected through a specific time change, we will get a closed form of the Laplace transform of the corresponding law for SBM.
Explicitly inverting the Laplace transform and taking next the derivative with respect to $y$, we will deduce the trivariate distribution for SBM started at the origin.
Finally conditioning with respect to the hitting time, we ultimately determine the full trivariate distribution started from any arbitrary $x\in\R$.

The plan of the paper is as follows:
In Section 2 we recall very briefly a few basic properties of SBM.
The third section contains the heavy work, namely the computation of the Laplace transform of the joint law $\PP_0(S_t\geq y,\,\G_t\in\rd\tau,\,L_t\in \rd l)$.
Finally in Section 4 we prove the main results, Theorems \ref{theo:trivariate_x} and \ref{theo:bivariate_reflected_x}.
%%%%%%%%%%%%%%%%%%%%%%%%%%%%%%%%%%%%%%%%%%%%%%%%%%%%%%%%%%%%%%%%%%%%%%%%%%%%%%%%%%%%%%%%%%%%%%%%%%%%%%%%
\section{Sticky Brownian motion and its properties}
\label{sec:SBM}
In this section we collect a few standard properties of SBM to be used in the sequel.
We refer e.g. to \cite{anagnostakis2022path,Howitt2007} and references therein for more details.
Following Feller \cite{feller1952parabolic}, SBM $(S_t)_{t\geq 0}$ can be defined as a Markov process with infinitesimal generator
$$
L=\dfrac{1}{2}\dfrac{d^2}{dx^2}
$$
and domain
$$
\operatorname{dom} (L)= \left\{f\in C_b (\R) :\qquad   f\in C^2 (\R^\ast)\qtext{and}  \dfrac{f'(0^+) - f'(0^-)}{\theta}=f''(0^-)=f''(0^+)  \right\}.
$$
Here we shall not dwelve in the rigorous probabilistic setting.
Let us just mention that $\left(\O,(\mathcal F_t)_{t\geq 0},\PP\right)$ is a usual filtered space, and we write as always $\PP_x,\EE_x$ to emphasize the initial condition $S_0=x$.

An alternate characterisation of SBM is that $(S_t)_{t\geq 0}$ solves the coupled SDE system
\begin{equation}
\label{eq:def_SBM_via_SDE}
\begin{cases}
\rd S_t =\indic_{[S_t \neq 0]} \rd B_t
\\
\indic_{[S_t =0]} \rd t =\theta \rd L_t  (S)
\end{cases},
\end{equation}
see \cite{engelbert2014stochastic}.
We stress that \eqref{eq:def_SBM_via_SDE} is nonstandard in that it does not admit strong solutions, see again \cite{engelbert2014stochastic}.
Roughly speaking, SBM is a standard Brownian motion outside of the origin, and the dynamics at $x=0$ is characterized by the fact that the time spent there is proportional to the local time as in \eqref{eq:Lt=Lt}.
As such, it is clear that SBM and BM share the same excursions, as long as crossing the origin is not involved.
In particular SBM and BM share the same first hitting time at the origin, which we denote by
$$
T_0(S)=T_0(B).
$$
An important property that we will use repeatedly in the remaining of this note is that SBM $S_t$ can be explicitly related to standard BM $B_t$ through an implicit time change \cite{ito1963brownian,engelbert2014stochastic}.
More precisely, setting
$$
A_t\coloneqq \int_0^t\indic_{[S_r\neq 0]}\rd r = t-\int_0^t\indic_{[S_r =0]}\rd r = t - \frac{1}{\theta}L_t(S),
$$
there holds
$$
S_t=B_{A_t}.
$$
As a consequence the SBM and BM local time and positive occupation time are also related by
 \begin{equation}
 \label{eq:Lt_Gt}
 L_t(S)=L_{A_t}(B)
 \qqtext{and}
 \G_t(S)=\G_{A_t}(B)+\frac 1\theta L_{A_t}(B),
 \end{equation}
 see e.g. \cite[lemma 3.1]{T21}.
 In addition, the pseudo inverse
$$
K_\tau
\coloneqq (A^{-1})_t=\inf\{t:\quad A_t>\tau\}
$$
satisfies
\begin{equation}
\label{eq:Ktau_dKtau}
 K_\tau = \tau +\frac{1}{\theta}L_\tau(B)
\qtext{and}
\rd K_\tau =\rd \tau +\frac{1}{\theta}\rd L_\tau(B).
\end{equation}

%%%%%%%%%%%%%%%%%%%%%%%%%%%%%%%%%%%%%%%%%%%%%%%%%%%%%%%%%%%%%%%%%%%%%%%%%%%%%%%%%%%%%%%%%%%%%%%%%%%%%%%%
\section{Trivariate Laplace transform}
The aim of this section is to compute the $(t,\tau,l)$ Laplace transform of the joint law
$$
\PP_0(S_t\geq y,\,\G_t\in\rd\tau,\,L_t\in \rd l).
$$
 We first compute this Laplace transform for pure Brownian motion:
\begin{lem}
\label{lem:integral_dt}
Fix $y\in \R$ and let $B_t$ be a standard Brownian motion with local time $L_t=L_t (B)$  and occupation time $\Gamma_t=\Gamma_t(B)=\int_0^t \indic_{[0,\infty)}(B_\tau)\rd\tau$.
Then for $\lambda,\beta,\gamma>0$ we have
\begin{multline}
\label{eq:u0_explicit}
 \EE_0\int_0^{\infty}\indic_{[y,\infty)}(B_t)e^{-\lambda t -\beta\Gamma_t-\gamma L_t}\rd t
 \\=
 \frac 2{2\g+\sqrt{2(\l+\b)}+\sqrt{2\l}}
\begin{cases}
        \frac 1{\sqrt{2(\l+\b)}}e^{-\sqrt{2(\l+\b)}y} & \mbox{if }y>0\\
       \frac 1{\sqrt{2\l}}\left(1-e^{\sqrt{2\l}y}\right)+\frac 1{\sqrt{2(\l+\b)}} & \mbox{if }y\leq 0.
\end{cases}
\end{multline}
\end{lem}
\noindent
We stress that the right-hand side in \eqref{eq:u0_explicit} is continuous across $y=0$.
For $y\geq 0$ this exact statement can be found in \cite[lemma 2.1]{karatzas1984trivariate}.
The case $y<0$ is addressed in \cite[lemma 3.1]{APP11}, but contains a mistake because the expression given therein is inconsistent with the limit $y\to-\infty$.
For the sake of completeness, and since this is actually the cornerstone of the paper, we chose to give here a self-contained proof with full details.
\begin{proof}
Let
$$
u(x)
\coloneqq
\EE_x\int_0^{\infty}\indic_{[y,\infty)}(B_t)e^{-\lambda t -\b\G_t-\gamma L_t}\rd t.
$$
We follow \cite[\S 2.3]{ito1996diffusion} and derive below an ODE satisfied by $u$, solve it explicitly, and finally evaluate $u(0)$.

We first claim that $u$ is continuous and bounded, $C^1$ except at $x=0$ (but in particular $C^1$ at $x=y$), $C^2$ outside of $x\in\{0,y\}$ with
\begin{equation}
\label{eq:ODE_u}
 \left[\lambda+\beta\, \indic_{[0,\infty)}(x)\right] u(x) -\frac 12 u''(x)=\indic_{[y,\infty)}(x),
\end{equation}
and satisfies the slope condition
\begin{equation}
\frac{1}{2}\left(u'(0^+)-u'(0^-)\right)=\gamma u(0).
\label{eq:slope_compatibility}
\end{equation}
As a first step, boundedness is obvious from $0\leq \indic_{[y,\infty)}(B_t)e^{-\lambda t -\b\G_t-\gamma L_t}\leq e^{-\lambda t}$.
Let now $T_0=T_0(B)=\min\{t>0:B_t=0\}$ denote the first hitting time at zero, and write
\begin{multline*}
u(x)
=
\EE_x\int_0^{T_0}\indic_{[y,\infty)}(B_t)e^{-\lambda t -\b\G_t-\gamma L_t}\rd t
+
\EE_x\int_{T_0}^{\infty}\indic_{[y,\infty)}(B_t)e^{-\lambda t -\b\G_t-\gamma L_t}\rd t
\\
=
\underbrace{
\EE_x\int_0^{T_0}\indic_{[y,\infty)}(B_t)e^{-\lambda t -\b\G_t}\rd t
}_{\coloneqq v(x)}
+
\underbrace{
\EE_x\left(e^{-\lambda T_0}\int_0^{\infty}\indic_{[y,\infty)}(B_{s+T_0})e^{-\lambda s -\b\G_{s+T_0}-\gamma L_{s+T_0}}\rd s\right)}_{\coloneqq w(x)}.
\end{multline*}
Here we used that the local time $L_t=0$ for $t\leq T_0$.

Let us compute separately $v(x)$ and $w(x)$.
By definition of the occupation time we have $\Gamma_t=t$ for $t\leq T_0$ if $x\geq 0$, while $\Gamma_t=0$ for $t\leq T_0$ if $x<0$.
Hence
\begin{equation}
\label{eq:def_v_x_><0}
v(x)=
\begin{cases}
\EE_x\int_0^{T_0}\indic_{[y,\infty)}(B_t)e^{-(\lambda+\beta) t}\rd t & \mbox{if }x\geq 0
\\
\EE_x\int_0^{T_0}\indic_{[y,\infty)}(B_t)e^{-\lambda t}\rd t & \mbox{if }x< 0.
\end{cases}
\end{equation}
For $x<0$ we have thus, by the strong Markov property,
\begin{multline*}
v(x)
=
\EE_x\int_0^{T_0}\indic_{[y,\infty)}(B_t)e^{-\lambda t}\rd t
=
\EE_x\int_0^{\infty}\indic_{[y,\infty)}(B_t)e^{-\lambda t}\rd t  - \EE_x\int_{T_0}^{\infty}\indic_{[y,\infty)}(B_t)e^{-\lambda t}\rd t
\\
=
\int_0^{\infty}\EE_x\left[\indic_{[y,\infty)}(B_t)\right]e^{-\lambda t}\rd t - \EE_x\left(e^{-\lambda{T_0}}\int_{0}^{\infty}\indic_{[y,\infty)}(B_{s+{T_0}})e^{-\lambda s}\rd s\right)
\\
=
\int_0^{\infty}\EE_x\left[\indic_{[y,\infty)}(B_t)\right]e^{-\lambda t}\rd t  - \left(\EE_xe^{-\lambda{T_0}}\right)\EE_x\left(\int_{0}^{\infty}\indic_{[y,\infty)}(B_{s+{T_0}})e^{-\lambda s}\rd s\right)
\\
=
\int_0^{\infty}\EE_x\left[\indic_{[y,\infty)}(B_t)\right]e^{-\lambda t}\rd t  - \left(\EE_xe^{-\lambda{T_0}}\right)\EE_0\left(\int_{0}^{\infty}\indic_{[y,\infty)}(B_{s})e^{-\lambda s}\rd s\right)
\\
=
\int_0^{\infty}\EE_x\left[\indic_{[y,\infty)}(B_t)\right]e^{-\lambda t}\rd t  - \left(\EE_xe^{-\lambda{T_0}}\right)\int_{0}^{\infty}\EE_0\left[\indic_{[y,\infty)}(B_{s})\right]e^{-\lambda s}\rd s.
\end{multline*}
Now, since by definition $T_0\sim h$ one has $\EE_x e^{-\lambda T_0}=\int_0^\infty e^{-\lambda t}h(t,x)\rd t=e^{-\sqrt{2\lambda}|x|}$, whence
$$
x<0:\qquad
v(x)
=
\int_0^{\infty}\int_y^{\infty}g_t(z-x)\rd ze^{-\lambda t}\rd t  - e^{-\sqrt{2\lambda}|x|}\int_{0}^{\infty}\int_y^{\infty}g_s(z)\rd ze^{-\lambda s}\rd s.
$$
Similarly, changing $\lambda$ into $\lambda+\beta$ in \eqref{eq:def_v_x_><0} gives
\begin{multline*}
x\geq 0:\qquad
v(x)
=
\int_0^{\infty}\int_y^{\infty}g_t(z-x)\rd ze^{-(\lambda+\beta) t}\rd t
\\
- e^{-\sqrt{2(\lambda+\beta)}|x|}\int_{0}^{\infty}\int_y^{\infty}g_s(z)\rd ze^{-(\lambda+\beta) s}\rd s.
\end{multline*}
For $w$ we have similarly
\begin{multline*}
 w(x)=\EE_x\left(e^{-\lambda{T_0}}\int_0^{\infty}\indic_{[y,\infty)}(B_{s+{T_0}})e^{-\lambda s -\b\G_{s+{T_0}}-\gamma L_{s+{T_0}}}\rd s\right)
 \\
 =\EE_x\left(e^{-\lambda{T_0}}\right)\left(\EE_0\int_0^{\infty}\indic_{[y,\infty)}(B_{s})e^{-\lambda s -\b\G_{s}-\gamma L_{s}}\rd s\right)
 = e^{-\sqrt{2\lambda}|x|} u(0).
\end{multline*}
With this explicit representation formula for $u(x)=v(x)+w(x)$ it is then a simple exercise to establish the regularity of $u$ and differentiate under the integral sign (with $\partial^2_{xx}g_t=2\partial_t g_t$) in order to establish \eqref{eq:ODE_u}\eqref{eq:slope_compatibility}.

In order to get to \eqref{eq:u0_explicit} it suffices now to solve the ODE problem explicitly.
Consider first $y< 0$:
Solving explicitly \eqref{eq:ODE_u} on the three separate intervals we have
$$
u(x)=
\begin{cases}
  C_1\exp\left(\sqrt{2\l} x\right) & \mbox{if }x\in (-\infty,y)\\
  C_2\exp\left(\sqrt{2\l} x\right) + C_3\exp\left(-\sqrt{2\l} x\right)+\frac{1}{\l} & \mbox{if }x\in(y,0)\\
  C_4 \exp\left(-\sqrt{2(\l+\b)} x\right)+\frac{1}{\l+\b} & \mbox{if }x\in(0,\infty).
\end{cases}
$$
($u$ is bounded so one of the exponential solutions must be discarded at infinity.)
The four constants $C_i,i=1\dots 4$ can be explicitly determined by solving a linear system arising from the four conditions that $u$ is continuous across $x=0$ and $x=y$, $u$ is $C^1$ across $x=y$, and \eqref{eq:slope_compatibility}.
Here we omit the elementary but tedious computations for the sake of brevity.
Finally evaluating $u(0)=C_4+\frac{1}{\l+\b}$ gives the second case in \eqref{eq:u0_explicit}.

For $y>0$ a very similar computation leads to the first alternative in \eqref{eq:u0_explicit}, as correctly derived in \cite[lem. 3.1]{APP11}.

Finally, for the particular case $y=0$ one can simply argue that, for fixed $x$, the map $y\mapsto \EE_x\int_0^{\infty}\indic_{[y,\infty)}(B_t)e^{-\l t -\b\Gamma_t-\gamma L_t}\rd t$ is continuous (by Lebesgue dominated convergence, $y\mapsto \indic_{[y,\infty)}(B_t(\omega))$ being continuous for $\rd \PP\rd t$ a.e. $(\omega,t)$ by standard properties of Brownian motion).
\end{proof}
As already anticipated we can now use the time change to compute the exact same quantity for SBM instead of pure BM:
\begin{prop}
\label{cor:E0_sticky}
Let $(S_t)_{t\geq 0}$ be a standard SBM with local time at the origin $L_t=L_t (S)$ and occupation time $\Gamma_t=\G_t(S)$.
For fixed $y\in \R$ there holds
\begin{multline}
\label{eq:E0_sticky}
\EE_0\int_0^\infty \indic_{[y,\infty)}(S_t) e^{-\lambda t -\beta\G_t-\gamma L_t} dt
\\
=
\frac 2{2\left(\g+\frac{\l+\b}\theta\right)+\sqrt{2(\l+\b)}+\sqrt{2\l}}
\begin{cases}
        \frac 1{\sqrt{2(\l+\b)}}e^{-\sqrt{2(\l+\b)}y} & \mbox{if }y>0\\
       \frac 1{\sqrt{2\l}}\left(1-e^{\sqrt{2\l}y}\right)+\frac 1{\sqrt{2(\l+\b)}} +\frac 1\theta & \mbox{if }y\leq 0
\end{cases}
\end{multline}
\end{prop}
\noindent
Note that this is almost identical to \eqref{eq:u0_explicit} for pure Brownian motion, upon substituting $\gamma$ with $\tilde\g =\g+\frac\b\theta+\frac\l\theta$.
However, a significant difference is the last $\frac 1\theta$ term for $y\leq 0$, making now \eqref{eq:E0_sticky} discontinuous across $y=0$.
This will ultimately lead to the $\delta_0(\rd y)$ terms in \eqref{eq:trivariate_x} and \eqref{eq:bivariate_reflected_x}, which are precisely the terms missing in \cite{T21}.
\begin{proof}
Changing time variable $t=K_\tau$ and $A_t=(K^{-1})_t=\tau$ with $L_t(S)=L_{A_t}(B)=L_\tau(B)$ as in section~\ref{sec:SBM}, and exploiting \eqref{eq:Lt_Gt}\eqref{eq:Ktau_dKtau}, we have
\begin{multline*}
I
\coloneqq
\EE_0\int_0^\infty  \indic_{[y,\infty)}(S_t)e^{-\lambda t -\beta \G_t(S)-\gamma L_t(S)} \rd t
\\
=\EE_0\int_0^\infty  \indic_{[y,\infty)}(B_{A_t}) e^{-\lambda t -\beta\left(\G_{A_t}(B)+\frac 1\theta L_{A_t}(B)\right)-\gamma L_{A_t}(B)} \rd t
\\
=\EE_0\int_0^\infty  \indic_{[y,\infty)}(B_{\tau}) e^{-\lambda K_\tau -\beta\Gamma_\tau(B)-\left(\gamma+\frac\b\theta\right) L_{\tau}(B)} \rd K_\tau
\\
=\EE_0\int_0^\infty  \indic_{[y,\infty)}(B_{\tau}) e^{-\lambda \tau -\beta\G_\tau(B)-(\gamma+\frac \b\theta +\frac\l\theta) L_{\tau}(B)} \left(\rd \tau +\frac{1}{\theta}\rd L_\tau(B)\right)
\\
=\underbrace{\EE_0\int_0^\infty  \indic_{[y,\infty)}(B_{\tau}) e^{-\lambda \tau -\b\G_\tau(B)-\tilde\g L_{\tau}(B)} \rd \tau}_{\coloneqq I_1}
+
\underbrace{\frac 1\theta\EE_0\int_0^\infty  \indic_{[y,\infty)}(B_{\tau}) e^{-\lambda \tau -\b\G_\tau(B)-\tilde\g L_{\tau}(B)}\rd L_\tau(B)}_{\coloneqq I_2},
\end{multline*}
where we put $\tilde\g=\g+\frac\b\theta+\frac\l\theta$.
Note that the latter expression only involves pure Brownian motion and its local and occupation times $L_\tau(B),\G_\tau(B)$.
The first term is precisely given by Lemma~\ref{lem:integral_dt}, namely
$$
I_1=\frac 2{2\tilde\g+\sqrt{2(\l+\b)}+\sqrt{2\l}}
\begin{cases}
        \frac 1{\sqrt{2(\l+\b)}}e^{-\sqrt{2(\l+\b)}y} & \mbox{if }y>0\\
       \frac 1{\sqrt{2\l}}\left(1-e^{\sqrt{2\l}y}\right)+\frac 1{\sqrt{2(\l+\b)}} & \mbox{if }y\leq 0.
\end{cases}
$$
For $I_2$ we observe that, by classical properties of Brownian motion, the time measure $\rd L_\tau(B)$ only charges $\{\tau:\,B_\tau=0\}$, so we can first remove the indicator as
$$
I_2=
\begin{cases}
0 & \mbox{if }y>0\\
\frac 1\theta\EE_0\int_0^\infty  e^{-\lambda \tau -\b\G_\tau(B)-\tilde\g L_{\tau}(B)}\rd L_\tau(B) & \mbox{if }y\leq 0
\end{cases}
.
$$
In order to compute
$$
J\coloneqq \EE_0\int_0^\infty  e^{-\lambda \tau -\b\G_\tau(B)-\tilde\g L_{\tau}(B)}\rd L_\tau(B)
$$
we take a detour and first determine
\begin{multline*}
 J_1
\coloneqq \EE_0\int_0^\infty  e^{-\lambda \tau -\b\G_\tau(B)-\tilde\g L_{\tau}(B)}[\b\rd \G_\tau(B) + \tilde\g\rd L_\tau(B)]
 \\
 = - \EE_0\int_0^\infty  e^{-\lambda \tau}\rd\left(e^{-\b\G_\tau(B)-\tilde\g L_{\tau}(B)}\right)
 \\
 = - \EE_0\left([0-1] +  \int_0^\infty  e^{-\b\G_\tau(B)-\tilde\g L_{\tau}(B)} \l e^{-\lambda \tau}\rd\tau\right)
 \\
 = 1 -\l \EE_0\int_0^\infty  e^{-\lambda \tau -\b\G_\tau(B)-\tilde\g L_{\tau}(B)}\rd\tau.
\end{multline*}
Here we leveraged $L_0=\Gamma_0=0$ and $L_\infty=\Gamma_\infty=\infty$ almost surely for Brownian motion.
Taking the limit $y\to-\infty$ in \eqref{eq:u0_explicit}, an easy dominated convergence allows to retrieve
\begin{multline*}
J_1=
1- \lambda \lim\limits_{y\to-\infty}\EE_0\int_0^\infty  \indic_{[y,\infty)}(B_\tau)e^{-\lambda \tau -\b\G_\tau(B)-\tilde\g L_{\tau}(B)}\rd\tau
\\
=1 -\l \frac 2{2\tilde \g+\sqrt{2(\l+\b)}+\sqrt{2\l}} \left(\frac 1{\sqrt{2\l}}+\frac 1{\sqrt{2(\l+\b)}}\right).
\end{multline*}
By definition of occupation time we have $\rd \G_\tau(B)=\indic_{[0,\infty)}(B_\tau)\rd \tau$, thus we can also compute explicitly
\begin{multline*}
J_2\coloneqq \EE_0\int_0^\infty  e^{-\lambda \tau -\b\G_\tau(B)-\tilde\g L_{\tau}(B)}\rd \G_\tau(B)
\\
=
\EE_0\int_0^\infty  \indic_{[0,\infty)}(B_\tau) e^{-\lambda \tau -\b\G_\tau(B)-\tilde\g L_{\tau}(B)}\rd\tau
\\
=
\frac 2{2\tilde\g+\sqrt{2(\l+\b)}+\sqrt{2\l}}\frac 1{\sqrt{2(\l+\b)}},
\end{multline*}
again from Lemma~\ref{lem:integral_dt} (with $y=0$).
Straightforward algebra therefore leads to
$$
 J=\frac 1{\tilde\g}\left(J_1-\b J_2\right)
 = \frac 2{2\tilde\g+\sqrt{2(\l+\b)}+\sqrt{2\l}}.
$$
Finally, we get
\begin{multline*}
I= I_1+I_2=I_1+
\begin{cases}
 0 & \mbox{if }y>0\\
 \frac 1\theta J &\mbox{if }y\leq 0
\end{cases}
\\
=
\frac 2{2\tilde\g+\sqrt{2(\l+\b)}+\sqrt{2\l}}
\begin{cases}
        \frac 1{\sqrt{2(\l+\b)}}e^{-\sqrt{2(\l+\b)}y} & \mbox{if }y>0\\
       \frac 1{\sqrt{2\l}}\left(1-e^{\sqrt{2\l}y}\right)+\frac 1{\sqrt{2(\l+\b)}} +\frac 1\theta & \mbox{if }y\leq 0
\end{cases}
\end{multline*}
as claimed and the proof is complete.
\end{proof}
%%%%%%%%%%%%%%%%%%%%%%%%%%%%%%%%%%%%%%%%%%%%%%%%%%%%%%%%%%%%%%%%%%%%%%%%%%%%%%%%%%%
\section{Technical statements and proofs}
Our goal in this section is to prove Theorems \ref{theo:trivariate_x} and \ref{theo:bivariate_reflected_x}.
% In view of Proposition \ref{cor:E0_sticky},
% to obtain the trivariate of the SBM started at zero, we only have to compute the inverse Laplace transform of the right-hand-side of \eqref{eq:E0_sticky} and taking the derivative with respect to $y$. The expression of the trivariate for the SBM started at an arbitrary point $x$ namely Theorem \ref{theo:trivariate_x} will then be obtained by conditioning and using that the excursions of SBM are the same as BM. Corollary \ref{theo:bivariate_reflected_x} will follow by integrating the trivariate in $\tau$.
% We begin by computing the inverse Laplace transform of \eqref{eq:E0_sticky}.
\begin{prop}
\label{prop:trivariate_0}
The SBM started from $x=0$ has trivariate distribution
\begin{multline}
 \label{eq:trivariate_x=0}
 \PP_0(S_t\in\rd y,\G_t\in\rd\tau,\,L_t\in \rd l)
 =
\frac 1\theta h\left(t-\tau,\frac l2\right) h\left(\tau-\frac l\theta,\frac l2\right)\delta_0(\rd y)\rd\tau\rd l
\\
+ \rd y \rd\tau\rd l
\begin{cases}
h\left(t-\tau,\frac l2\right)h\left(\tau-\frac l\theta,\frac l2+y\right) & \mbox{if }y\geq 0
\\
h\left(\tau-\frac l\theta,\frac l2\right)h\left(t-\tau,\frac l2-y\right) & \mbox{if }y\leq 0
\end{cases}.
\end{multline}
\end{prop}
%
% \noindent
%
%
\begin{proof}
In Proposition~\ref{cor:E0_sticky} we actually computed explicitly the $\L_l\L_\tau\L_t$ Laplace transform of the joint law
$$
\PP_0(S_t\geq y,\,\G_t\in\rd\tau,\,L_t\in \rd l),
$$
with Laplace variables $(\l,\b,\g)$ dual to $(t,\tau,l)$.
The first step is therefore to invert the Laplace transform \eqref{eq:E0_sticky}.
For this we recall that the transforms of \eqref{eq:def_h} are
\begin{align*}
 G(s,x)
 &
 \coloneqq \L_t[g(t,x)](s)=\frac1{\sqrt{2s}}e^{-\sqrt{2s}|x|}
 \\
 H(s,x)
  & \coloneqq \L_t[h(t,x)](s)=e^{-\sqrt{2s}|x|}.
\end{align*}
For $y>0$ we first compute, by usual properties of the Laplace transform,
\begin{multline*}
\L_l\L_\tau\L_t\left[h\left(t-\tau,\frac l2\right)g\left(\tau-\frac l\theta,\frac l2 + y\right)\right](\l,\b,\g)
=
\\
\L_l\L_\tau\left[e^{-\l\tau}H\left(\l,\frac l2\right)g\left(\tau-\frac l\theta,\frac l2 + y\right)\right](\b,\g)
\\
=
\L_l\left[H\left(\l,\frac l2\right)\L_\tau\left[g\left(\bullet-\frac l\theta,\frac l2 + y\right)\right](\b+\l)\right](\g)
\\
=
\L_l\left[H\left(\l,\frac l2\right)e^{-\frac{l}{\theta}(\b+\l)}G\left(\b+\l,\frac l2+y\right)\right](\g)
\\
=
\L_l\left[e^{-\sqrt{2\l}\frac l 2}e^{-\frac{l}{\theta}(\b+\l)}\frac{e^{-\sqrt{2(\b+\l)}\left(\frac l 2 + y\right)}}{\sqrt{2(\b+\l)}}\right](\g)
\\
=\frac{e^{-\sqrt{2(\b+\l)}y}}{\sqrt{2(\b+\l)}}\L_l\left[e^{-l\left(\frac{\sqrt{2\l}}{2} + \frac{\b+\l}{\theta}+\frac{\sqrt{2(\b+\l)}}{2}\right)}\right](\g)
\\
=\frac{e^{-\sqrt{2(\b+\l)}y}}{\sqrt{2(\b+\l)}}
\times \frac{1}{\g + \left(\frac{\sqrt{2\l}}{2} + \frac{\b+\l}{\theta}+\frac{\sqrt{2(\b+\l)}}{2}\right)}.
\end{multline*}
This is exactly the right-hand side of \eqref{eq:E0_sticky} and therefore
\begin{equation}
\label{eq:P0_trivariate_y>0}
y>0:
\qquad
\PP_0(S_t\geq y,\,\G_t\in\rd\tau,\,L_t\in \rd l)=h\left(t-\tau,\frac l2\right) g\left(\tau-\frac l\theta,\frac l2 + y\right)\rd\tau\rd l.
\end{equation}
For $y\leq 0$ we split the right-hand side of \eqref{eq:E0_sticky} as $A\eqqcolon A_1-A_2+A_3+A_4$ and compute separately
\begin{multline*}
\L_l\L_\tau\L_t\left[h\left(\tau-\frac l\theta,\frac l2\right)g\left(t-\tau,\frac l 2\right)\right](\l,\b,\g)
\\
=
\L_l\L_\tau \left[h\left(\tau-\frac l\theta,\frac l2\right)e^{-\l\tau}G\left(\lambda,\frac l2\right)\right](\b,\g)
\\
=
\L_l\left[G\left(\lambda,\frac l2\right)\L_\tau\left[h\left(\bullet-\frac l\theta,\frac l2\right)\right](\l+\beta)\right](\g)
\\
=
\L_l\left[G\left(\lambda,\frac l2\right) e^{-\frac l\theta(\l+\b)} H\left(\l+\b,\frac l2\right)\right](\g)
\\
=\L_l\left[\frac{e^{-\sqrt{2\l}\frac l2}}{\sqrt{2\l}} e^{-\frac l\theta(\l+\b)} e^{-\sqrt{2(\l+\b)}\frac l2}\right](\g)
\\
=
\frac{1}{\sqrt{2\l}}\L_l\left[e^{-l\left(\frac{\sqrt{2\l}}{2}+\frac{\l+\b}{\theta} + \frac{\sqrt{2(\l+\b)}}{2}\right)}\right](\g)
\\
=
\frac{1}{\sqrt{2\l}}\frac 1{\g + \left(\frac{\sqrt{2\l}}{2}+\frac{\l+\b}{\theta} + \frac{\sqrt{2(\l+\b)}}{2}\right)}
=A_1
\end{multline*}
\begin{multline*}
\L_l\L_\tau\L_t\left[h\left(\tau-\frac l\theta,\frac l2\right)g\left(t-\tau,\frac l 2-y\right)\right](\l,\b,\g)
\\
=
\L_l\L_\tau \left[h\left(\tau-\frac l\theta,\frac l2\right)e^{-\l\tau}G\left(\lambda,\frac l2-y\right)\right](\b,\g)
\\
=
\L_l\left[G\left(\lambda,\frac l2-y\right)\L_\tau\left[h\left(\bullet-\frac l\theta,\frac l2\right)\right](\l+\beta)\right](\g)
\\
=
\L_l\left[G\left(\lambda,\frac l2-y\right) e^{-\frac l\theta(\l+\b)} H\left(\l+\b,\frac l2\right)\right](\g)
\\
=\L_l\left[\frac{e^{-\sqrt{2\l}\left(\frac l2-y\right)}}{\sqrt{2\l}} e^{-\frac l\theta(\l+\b)} e^{-\sqrt{2(\l+\b)}\frac l2}\right](\g)
\\
=
\frac{e^{\sqrt{2\l}y}}{\sqrt{2\l}}
\L_l\left[e^{-l\left(\frac{\sqrt{2\l}}{2}+\frac{\l+\b}{\theta} + \frac{\sqrt{2(\l+\b)}}{2}\right)}
\right](\g)
\\
=
\frac{e^{\sqrt{2\l}y}}{\sqrt{2\l}}\frac 1{\g + \left(\frac{\sqrt{2\l}}{2}+\frac{\l+\b}{\theta} + \frac{\sqrt{2(\l+\b)}}{2}\right)}
=A_2
\end{multline*}
\begin{multline*}
 \L_l\L_\tau\L_t \left[ h\left(t-\tau,\frac l2\right) g\left(\tau-\frac l\theta,\frac l2\right)\right](\l,\b,\g)
 \\
 =
 \L_l\L_\tau \left[e^{-\l\tau}H\left(\l,\frac l2\right) g\left(\tau-\frac l\theta,\frac l2\right)\right](\b,\g)
 \\
 =
 \L_l\left[H\left(\l,\frac l2\right) \L_\tau \left[g\left(\bullet-\frac l\theta,\frac l2\right)\right](\l+\b)\right](\g)
 \\
 =
 \L_l\left[H\left(\l,\frac l2\right) e^{-\frac l\theta(\l+\b)}G\left(\l+\b,\frac l2\right)\right](\g)
 \\
 =
 \L_l\left[e^{-\sqrt{2\l}\frac l2} e^{-\frac l\theta(\l+\b)}\frac{e^{-\sqrt{2(\l+\b)}\frac l2}}{\sqrt{2(\l+\b)}}\right](\g)
 \\
 =\frac{1}{\sqrt{2(\l+\b)}}\L_l\left[e^{-l\left(\frac{\sqrt{2\l}}{2}+\frac{\l+\b}{\theta}+\frac{\sqrt{2(\l+\b)}}{2}\right)}\right](\g)
 \\
 \frac{1}{\sqrt{2(\l+\b)}} \frac{1}{\g+\left(\frac{\sqrt{2\l}}{2}+\frac{\l+\b}{\theta}+\frac{\sqrt{2(\l+\b)}}{2}\right)}
 =A_3
\end{multline*}
\begin{multline*}
\L_l\L_\tau\L_t \left[  h\left(t-\tau,\frac l 2\right)h\left(\tau-\frac l\theta,\frac l2\right)\right](\l,\b,\g)
\\
=
\L_l\L_\tau\left[  e^{-\l\tau} H\left(\l,\frac l 2\right)h\left(\tau-\frac l\theta,\frac l2\right)\right](\b,\g)
\\
=
\L_l\left[  H\left(\l,\frac l 2\right)\L_\tau\left[h\left(\bullet-\frac l\theta,\frac l2\right)\right](\l+\b)\right](\g)
\\
=\L_l\left[  H\left(\l,\frac l 2\right)e^{-\frac l\theta(\l+\b)}H\left(\l+\b,\frac l2\right)\right](\g)
\\
=
\L_l\left[  e^{-\sqrt{2\l}\frac l2}e^{-\frac l\theta(\l+\b)}e^{-\sqrt{2(\l+\b)}\frac l2}\right](\g)
\\
=
\L_l\left[  e^{-l\left(\frac{\sqrt{2\l}}{2}+\frac{\l+\b}{\theta}+\frac{\sqrt{2(\l+\b)}}{2}\right)}\right](\g)
\\
=\frac{1}{\g+\left(\frac{\sqrt{2\l}}{2}+\frac{\l+\b}{\theta}+\frac{\sqrt{2(\l+\b)}}{2}\right)}
=\theta A_4.
\end{multline*}
Gathering the inverse Laplace transform of all four terms $A_1-A_2+A_3+A_4$ we deduce
\begin{multline}
 \label{eq:P0_trivariate_y<=0}
 y\leq 0:
\qquad
\PP_0(S_t\geq y,\,\G_t\in\rd\tau,\,L_t\in \rd l)
\\
=
\Bigg[
h\left(\tau-\frac l\theta,\frac l2\right) g\left(t-\tau,\frac l2\right)
- h\left(\tau-\frac l\theta,\frac l2\right) g\left(t-\tau,\frac l2-y\right)
\\
+ h\left(t-\tau,\frac l2\right) g\left(\tau-\frac l\theta,\frac l2\right)
+\frac 1\theta h\left(t-\tau,\frac l2\right) h\left(\tau-\frac l\theta,\frac l2\right)
\Bigg]\rd \tau\rd l.
\end{multline}
The two expression \eqref{eq:P0_trivariate_y>0}\eqref{eq:P0_trivariate_y<=0} fully determine the joint density $\PP_0(S_t\geq y,\,\G_t\in\rd\tau,\,L_t\in \rd l)$ of the pair $(\G_t,L_t)$ on the event $S_t\geq y$, for any $y\in \R$.
In order to retrieve the full trivariate density it only remains to take the derivative with respect to $y$.
Note however that the first two terms $A_1-A_2$ in \eqref{eq:P0_trivariate_y<=0} cancel out by continuity as $y\to 0$, while the third term $A_3$ term matches the limit of \eqref{eq:P0_trivariate_y>0} as $y\to 0$.
As a consequence, and as already anticipated, the last $A_4$ term with weight $\frac 1\theta$ induces a discontinuity across $y=0$.
Taking the distributional derivative
$$
 \PP_0(S_t\in\rd y,\,\G_t\in\rd\tau,\,L_t\in \rd l)
=-\frac{d}{dy}\PP_0(S_t\geq y,\,\G_t\in\rd\tau,\,L_t\in \rd l)
$$
with $\partial_x g(t,x)=-h(t,x)$ gives exactly \eqref{eq:trivariate_x=0} and the proof is complete.
\end{proof}

We are now in position to prove Theorem \ref{theo:trivariate_x}.

\begin{proof}[Proof of Theorem~\ref{theo:trivariate_x}]
Fix $x>0$ (the case $x=0$ is exactly Proposition~\ref{prop:trivariate_0}) and let $T_0=T_0(S)$ be the first hitting time at the origin.
Since SBM shares its excursions with pure BM we have $T_0(S)=T_0(B)$ and we simply write $T_0$.
We distinguish cases $T_0\leq t$ and $T_0>t$.
In this latter case observe that $L_t=0$ and $\Gamma_t=t$, while $S_t=B_t$ for $t<T_0$.
Hence by conditioning and the strong Markov property
 \begin{multline*}
\PP_x(S_t\in\rd y,\G_t\in\rd\tau,L_t\in \rd l)
\\
=
\PP_x(S_t\in\rd y,\G_t\in\rd\tau,L_t\in \rd l,T_0\leq t)
+ \PP_x(S_t\in\rd y,\G_t\in\rd\tau,L_t\in \rd l,T_0 > t)
\\
=\int_0^t \PP_x(S_t\in\rd y,\G_t\in\rd\tau,L_t\in \rd l\,\big \vert\, T_0 =s) \PP_x(T_0\in \rd s)
+\PP_x(B_t\in\rd y,t\in\rd\tau,0\in \rd l,T_0 > t)
\\
=\int_0^t \PP_0(S_{t-s}\in\rd y,\G_{t-s}\in\rd\tau,L_{t-s}\in \rd l) \PP_x(T_0\in \rd s)
+\PP_x(B_t\in\rd y,T_0 > t)\delta_t(\rd\tau)\delta_0(\rd l)
\\
=\underbrace{\int_0^t \PP_0(S_{t-s}\in\rd y,\G_{t-s}\in\rd\tau,L_{t-s}\in \rd l) \PP_x(T_0\in \rd s)}_{\eqqcolon I}
+p^0_t(x,y)\rd y\delta_t(\rd\tau)\delta_0(\rd l).
 \end{multline*}
The $p^0_t(x,y)$ term is exactly the first term in \eqref{eq:trivariate_x}, so we only need to deal with the first integral $I$.
  By definition the hitting time is distributed as $ \PP_x(T_0\in \rd s)=h(s,x)\rd s$, and $\PP_0(S_{t-s}\in\rd y,\Gamma_{t-s}\in\rd\tau,L_{t-s}\in \rd l)$ is given explicitly from Proposition~\ref{prop:trivariate_0} as a product of shifted $h$ functions.
 Note from \eqref{eq:trivariate_x=0} that the $\PP_0$ integrand in $I$ has density with respect to $(\delta_0(\rd y)+\rd y)\rd\tau \rd l$, so clearly $I$ will also be absolutely continuous with respect to that same measure and it suffices to compute its density $\rho_t(y,\tau,l)$.

 Extending $h(s,x)\equiv 0$ for times $s\leq 0$ we can first artificially extend the above time integral to the whole real line $I=\int_0^t(\dots)\rd s=\int_\R(\dots)\rd s$, with as always the implicit notation $0\leq \frac l\theta\leq \tau\leq t$, and then it suffices to distinguish cases depending on $y$.
 We have first
 \begin{multline*}
y=0:\qquad
\rho_t(0,\tau,l)
= \int_0^t\left[\frac 1\theta h\left(t-s-\tau,\frac l2\right) h\left(\tau-\frac l\theta,\frac l2\right)\right] h(s,x)\rd s
\\
=\frac 1\theta h\left(\tau-\frac l\theta,\frac l2\right)\int_\R\left(t-s-\tau,\frac l2\right)  h(s,x)\rd s
\\
=\frac 1\theta h\left(\tau-\frac l\theta,\frac l2\right)
\left[ h\left(\bullet,\frac l2\right)\ast h\left(\bullet,x\right) \right](t-\tau)
\\
=\frac 1\theta h\left(\tau-\frac l\theta,\frac l2\right)h\left(t-\tau,\frac l2+x\right),
 \end{multline*}
 where we leveraged the explicit convolution rule $ h(\bullet,a)\ast h(\bullet,b)=h(\bullet,a+b)$ (Markov property for the first hitting time $T_0$).
 Similarly, there holds
 \begin{multline*}
y>0: \qquad
\rho_t(y,\tau,l) =\int_0^t \left[h\left(t-s-\tau,\frac l2\right) h\left(\tau-\frac l\theta,\frac l2+y\right)\right] h(s,x)\rd s
\\
=h\left(\tau-\frac l\theta,\frac l2+y\right) \int_\R h\left(t-s-\tau,\frac l2\right)  h(s,x)\rd s
\\
=
h\left(\tau-\frac l\theta,\frac l2+y\right)\left[ h\left(\bullet,\frac l2\right)\ast h\left(\bullet,x\right) \right](t-\tau)
\\
=
h\left(\tau-\frac l\theta,\frac l2+y\right)h\left(t-\tau,\frac l2+x\right)
 \end{multline*}
and finally
\begin{multline*}
y<0: \qquad
\rho_t(y,\tau,l) =\int_0^th\left(\tau-\frac l\theta,\frac l2\right)h\left(t-s-\tau,\frac l2-y\right) h(s,x)\rd s
\\
=
h\left(\tau-\frac l\theta,\frac l2\right)\int_\R h\left(t-s-\tau,\frac l2-y\right) h(s,x)\rd s
\\
=
h\left(\tau-\frac l\theta,\frac l2\right)\left[ h\left(\bullet,\frac l2-y\right)\ast h\left(\bullet,x\right) \right](t-\tau)
\\
=
h\left(\tau-\frac l\theta,\frac l2\right)h\left(t-\tau,\frac l2+x-y\right).
 \end{multline*}
\end{proof}
By taking the $(y,l)$ marginal in Theorem \ref{theo:trivariate_x} we obtain the bivariate distribution:
\begin{cor}
For the SBM started from $x\geq 0$ there holds
\begin{multline}
\label{eq:bivariate_x}
\PP_x(S_t\in\rd y,\,L_t\in \rd l)
\\
=
p^0_t(x,y)\rd y\delta_0(\rd l)
+
h\left(t-\frac l\theta, l+x+|y|\right)\,\rd y\rd l
+
\frac 1\theta h\left(t-\frac l\theta,l+x\right)\delta_0(\rd y)\rd l
\end{multline}
\end{cor}
\begin{proof}
Simply integrate \eqref{eq:trivariate_x} in $\tau$, exploiting again the convolution property $h(\bullet,a)\ast h(\bullet,b)=h(\bullet,a+b)$.
\end{proof}
We end with
\begin{proof}[Proof of Theorem~\ref{theo:bivariate_reflected_x}]
For $y>0$ we write $|S_t|\in \rd y\Leftrightarrow S_t\in\left\{\rd y\cup-\rd y\right\}$, and recalling that $p_t(x,y)\equiv 0$ for $y<0$ we have by \eqref{eq:bivariate_x}
\begin{multline*}
y>0:\qquad
\PP_x(|S_t|\in\rd y,\,L_t\in \rd l)
\\
=
\PP_x(S_t\in\rd y,\,L_t\in \rd l) + \PP_x(S_t\in-\rd y,\,L_t\in \rd l)
\\
=\left[p^0_t(x,y)\rd y\delta_0(\rd l) + h\left(t-\frac l\theta, l+x+|y|\right)\,\rd y\rd l\right]
\\
+ h\left(t-\frac l\theta, l+x+|-y|\right)\,\rd y\rd l
\\
=p^0_t(x,y)\rd y\delta_0(\rd l) + 2h\left(t-\frac l\theta, l+x+y\right)\,\rd y\rd l.
\end{multline*}
For $y=0$ only the last $\delta_0(\rd y)$ in \eqref{eq:bivariate_x} contributes, and \eqref{eq:bivariate_reflected_x} follows.
\end{proof}

\section*{Acknowledgments}
J.-B.C. was supported by FCT - Funda\c{c}\~{a}o para a Ci\^encia e a Tecnologia, under the project UIDB/04561/2020
L.M. was funded by FCT - Funda\c{c}\~{a}o para a Ci\^encia e a Tecnologia through a personal grant 2020/00162/CEECIND.
The authors wish to thank A. Anagnostakis and D. Villemonais for fruitful discussions.

\bibliographystyle{plain}
\bibliography{biblio}

\begin{thebibliography}{10}

\bibitem{amir1991sticky}
Madjid Amir.
\newblock Sticky brownian motion as the strong limit of a sequence of random
  walks.
\newblock {\em Stochastic processes and their applications}, 39(2):221--237,
  1991.

\bibitem{anagnostakis2022path}
Alexis Anagnostakis.
\newblock {\em Path-wise study of singular diffusions.}
\newblock PhD thesis, Universit{\'e} de Lorraine (Nancy), 2022.

\bibitem{APP11}
Thilanka Appuhamillage, Vrushali Bokil, Enrique Thomann, Edward Waymire, and
  Brian Wood.
\newblock {Occupation and local times for skew Brownian motion with
  applications to dispersion across an interface}.
\newblock {\em The Annals of Applied Probability}, 21(1):183 -- 214, 2011.

\bibitem{bou2020sticky}
Nawaf Bou-Rabee and Miranda~C Holmes-Cerfon.
\newblock Sticky brownian motion and its numerical solution.
\newblock {\em SIAM review}, 62(1):164--195, 2020.

\bibitem{casteras2022hidden}
Jean-Baptiste Casteras and L{\'e}onard Monsaingeon.
\newblock Hidden dissipation and convexity for kimura equations.
\newblock {\em arXiv preprint arXiv:2209.15361}, 2022.

\bibitem{CMN}
Jean-Baptiste Casteras, L{\'e}onard Monsaingeon, and Luca Nenna.
\newblock Sticky diffusions and optimal transport.
\newblock {\em In preparation}.

\bibitem{chalub2021gradient}
Fabio~ACC Chalub, L{\'e}onard Monsaingeon, Ana~Margarida Ribeiro, and Max~O
  Souza.
\newblock Gradient flow formulations of discrete and continuous evolutionary
  models: a unifying perspective.
\newblock {\em Acta Applicandae Mathematicae}, 171(1):24, 2021.

\bibitem{engelbert2014stochastic}
Hans-J{\"u}rgen Engelbert and Goran Peskir.
\newblock Stochastic differential equations for sticky brownian motion.
\newblock {\em Stochastics An International Journal of Probability and
  Stochastic Processes}, 86(6):993--1021, 2014.

\bibitem{feller1952parabolic}
William Feller.
\newblock The parabolic differential equations and the associated semi-groups
  of transformations.
\newblock {\em Annals of Mathematics}, pages 468--519, 1952.

\bibitem{feller1954diffusion}
William Feller.
\newblock Diffusion processes in one dimension.
\newblock {\em Transactions of the American Mathematical Society}, 77(1):1--31,
  1954.

\bibitem{feller1957generalized}
William Feller.
\newblock Generalized second order differential operators and their lateral
  conditions.
\newblock {\em Illinois journal of mathematics}, 1(4):459--504, 1957.

\bibitem{Howitt2007}
Christopher~John Howitt.
\newblock {\em Stochastic flows and Sticky Brownian motion.}
\newblock PhD thesis, University of Warwick, 2007.

\bibitem{ito1963brownian}
Kiyoshi It{\^o} and Henry~P McKean~Jr.
\newblock Brownian motions on a half line.
\newblock {\em Illinois journal of mathematics}, 7(2):181--231, 1963.

\bibitem{ito1996diffusion}
Kiyosi It{\^o}, P~Henry~Jr, et~al.
\newblock {\em Diffusion processes and their sample paths: Reprint of the 1974
  edition}.
\newblock Springer Science \& Business Media, 1996.

\bibitem{jordan1998variational}
Richard Jordan, David Kinderlehrer, and Felix Otto.
\newblock The variational formulation of the fokker--planck equation.
\newblock {\em SIAM journal on mathematical analysis}, 29(1):1--17, 1998.

\bibitem{karatzas1984trivariate}
Ioannis Karatzas and Steven~E Shreve.
\newblock Trivariate density of brownian motion, its local and occupation
  times, with application to stochastic control.
\newblock {\em The Annals of Probability}, pages 819--828, 1984.

\bibitem{konarovskyi2021spectral}
Vitalii Konarovskyi, Victor Marx, and Max von Renesse.
\newblock Spectral gap estimates for brownian motion on domains with
  sticky-reflecting boundary diffusion.
\newblock {\em arXiv preprint arXiv:2106.00080}, 2021.

\bibitem{monsaingeon2021new}
L{\'e}onard Monsaingeon.
\newblock A new transportation distance with bulk/interface interactions and
  flux penalization.
\newblock {\em Calculus of Variations and Partial Differential Equations},
  60(3):101, 2021.

\bibitem{peskir2015boundary}
Goran Peskir.
\newblock On boundary behaviour of one-dimensional diffusions: from brown to
  feller and beyond.
\newblock {\em William Feller, Selected Papers II}, pages 77--93, 2015.

\bibitem{T21}
Wajdi Touhami.
\newblock On skew sticky brownian motion.
\newblock {\em Statistics {\&} Probability Letters}, 173:109086, 2021.

\end{thebibliography}

\bigskip
\noindent
{\sc
Jean-Baptiste (\href{mailto:jeanbaptiste.casteras@gmail.com}{\tt jeanbaptiste.casteras@gmail.com}).
\\
CMAFcIO, Faculdade de Ci\^encias da Universidade de Lisboa, Edificio C6, Piso 1, Campo Grande 1749-016 Lisboa, Portugal
}
\\

\noindent
{\sc
L\'eonard Monsaingeon (\href{mailto:leonard.monsaingeon@univ-lorraine.fr}{\tt leonard.monsaingeon@univ-lorraine.fr}).
\\
Institut \'Elie Cartan de Lorraine, Universit\'e de Lorraine, Site de Nancy B.P. 70239, F-54506 Vandoeuvre-l\`es-Nancy Cedex, France
\\
Grupo de F\'isica Matem\'atica, GFMUL, Faculdade de Ci\^encias, Universidade de Lisboa, 1749-016 Lisbon, Portugal.
}

\end{document}